\documentclass[runningheads,envcountsame,a4paper]{llncs} 

\usepackage{amssymb,amsmath,mathtools}
\usepackage{wasysym}
\usepackage{graphicx}
\usepackage{epsfig}
\usepackage{latexsym}
\usepackage[english]{babel}
\usepackage[utf8]{inputenc}
\usepackage{csquotes}
\usepackage{algorithm} 
\usepackage{algpseudocode}
\usepackage{tikz}
\usetikzlibrary{arrows,shapes.geometric,positioning,calc,cd}
\usepackage[colorinlistoftodos]{todonotes}
\usepackage{changepage}
\usepackage{subfig}
\usepackage{soul}
\usepackage{hyperref}
\usepackage{todonotes}
\usepackage{enumerate}

\newtheorem{thm}{Theorem}
\newtheorem{prop}[thm]{Proposition}
\newtheorem{cor}[thm]{Corollary}
\newtheorem{lem}[thm]{Lemma}

\newtheorem{exa}[thm]{Example}

\newsavebox{\qedB}

\sbox{\qedB}{\setlength{\unitlength}{1mm}
 \begin{picture}(4,4)(0,0)
  \thinlines
  {\put(0,0){\framebox(2.83,2.83){}}}%
  {\put(1.17,1.17){\framebox(2.83,2.83){}}}%
  {\put(0,0){\framebox(4,4){}}}%
  {\put(1.17,1.17){{\rule{1ex}{1ex} }}}%
 \end{picture}}
 
\newcommand{\QEDB}{\ifmmode\def\next{\tag"\usebox{\qedB}"}%
 \else\let\next=\relax
 {\unskip\nobreak\hfil\penalty50
 \hskip2em\hbox{}\nobreak\hfil\usebox{\qedB}
 \parfillskip=0pt \finalhyphendemerits=0\penalty-100\bigskip}\fi\next}

\newcommand{\Alphabet}{\hbox{\rm Alph}}

\newcommand{\fac}{\hbox{\rm Fac}}

\newcommand{\N}{\mathbb N}

\newcommand{\bprop}{\begin{prop}}
\newcommand{\eprop}{\end{prop}}
\newcommand{\bcor}{\begin{cor}}
\newcommand{\ecor}{\end{cor}}
\newcommand{\blem}{\begin{lem}}
\newcommand{\elem}{\end{lem}}

\newcommand{\cl}{\mathrm{Class}}
\newcommand{\pri}{\mathrm{Prim}}

\newcommand{\qcl}{\mathbb{Q}\mathrm{Class}}

\newcommand{\QQ}{\mathbb{Q}}
\newcommand{\ZZ}{\mathbb{Z}}

\newcommand{\rp}{\mathrm{RP}}

\title{On the number of rational power factors in a finite word}
\toctitle{Square complexity}
\authorrunning{S. Li, Y. Song}
\tocauthor{Shuo Li, Yuan Song}
\author{S. Li\inst{1} \and Y. Song\inst{2}}

\institute{Hangzhou International Innovation Institute of Beihang University, \\Sino-French Laboratory for Mathematics\\
\email{shuoli@buaa.edu.cn}
\and
LMIB-School of Mathematical Sciences, Beihang University, Beijing\\
\email{yuan.song@buaa.edu.cn}
 }

\begin{document}
\maketitle


\begin{abstract}
Let $w$ be a finite word of length $n$. In this paper, we study the maximum possible number of distinct rational power factors in a finite word. A rational power is a word of the form $u=p^kp'$, where $p$ is a nonempty finite word, $k$ is an integer larger than $1$, $p^k$ is a concatenation of $k$ copies of $p$ and $p'$ is a prefix of $p$. The rational powers can be recognized as a generalization of $k$-powers, and it is proved in~\cite{kpower} that, the number $C_k(w)$ of distinct $k$-powers in $w$ satisfies $C_k(w) \leq \frac{n-1}{k-1}$. However, the number of rational powers has not been studied in the literature. In this article, we prove that the number $\rp(w)$ of distinct rational power factors of $w$ satisfies
\[
 \rp(w)\le
\frac18n^2+O(n).
\]
We also illustrate a novel approach to study pattern-counting problems: using a graph-theoretic representation of words and a few word equations, we transform the traditional pattern-counting problems into a constrained extremal problem.  
\end{abstract}

\section{Introduction}

 The study of repetitions is a long-standing and active topic in combinatorics on words. The first result concerning repetitions can be dated back to Thue~\cite{Thue1906}, who proved that, the simplest repetition - the square words, i.e. the words of the form $w=vv$, can be avoided in infinite words of three letters. 

Counting the number of repetitions in finite words was initiated by Fraenkel and Simpson, who proved that the number of distinct square factors in a word of length $n$ is at most $2n$ and conjectured that $C_2(w)\le n$, see~\cite{FraenkelS98}. This conjecture was proved by Brlek and the first author in~\cite{sq}.
Similar questions have also been studied for $k$-powers. Crochemore et al. proved in ~\cite{CMFS10} that for a fixed integer $k\ge 3$, the number of distinct $k$-powers in a word of length $n$ is at most $\frac{n}{k-2}$. Pachocki, Radoszewski and the first author proved in \cite{kpower} that for a fixed integer \(k\ge2\), the maximum number of distinct nonempty \(k\)-power factors in a word of length
\(n\) is between
\[
        \frac{n}{k-1}-\Theta(\sqrt n)
        \quad\text{and}\quad
        \frac{n-1}{k-1}.
\]

These results show that the counting function of $k$-powers increases linearly in terms of the length of the word. An easy observation shows that the number of rational powers can increase quadratically in terms of length, see for example the words of the form $(ab)^k$. 

Another related work is about closed-word countings. In~\cite{OS24}, Olga and Svetlana proved that the number $C_{closed}(w)$ of closed words, i.e. the words of the form $aba$, where $a$ is a nonempty word, in $w$, satisfies that $C_{closed}(w)\leq \frac{n^2}{6}$ and that this upper bound is sharp. Since all rational powers are closed words, we have naturally the upper bound $\rp(w) \leq \frac{n^2}{6}$. The strategy used in~\cite{OS24} is to consider how many new closed words can be created by adding one single letter at the end of a string. The same strategy has also been applied by Fraenkel and Simpson in~\cite{FraenkelS98} to prove $C_2(w) \leq 2n$. In \cite{sq} and \cite{kpower}, to give a tighter upper bound to $C_2(w)$ as well as $C_k(w)$, the first author and his co-authors used a graph-theoretic method and constructed an injection from $k$-powers in $w$ to the small circuits in the Rauzy graphs of $w$. In this work, using the same correspondence between $k$-powers and small circuits and a few more word equations, we transform the rational-power-counting problems into a constrained extremal problem and prove that $\rp(w) \leq \frac{n^2}{8}$. 

The paper is organized as follows. 
Section~\ref{sec2} recalls basic terminology on words, primitive words, conjugacy classes, rational powers, and Rauzy graphs. 
Section~\ref{sec3} gives some basic facts of rational power factors. 
Section~\ref{sec4} proves the local estimates for one primitive conjugacy class, establishes the restrictions on all the classes, and completes the optimization leading to the main theorem.

\section{Preliminaries}\label{sec2}

Let us recall the basic terminology about words from Lothaire \cite{lothaire3}. A {\em word} is a finite sequence   $w = w_1 w_2 \cdots w_{n}$ of \emph{letters} or symbols. The {\em length} $|w|$ of $w$ is $n$ and $w_i$ is the letter in  {\em position} $i$. 
The {\em concatenation of $w= w_1 w_2 \cdots w_{n}$ and $v= v_1 v_2 \cdots v_{m}$} is defined as 
\[wv=w_1 w_2 \cdots w_{n}v_1 v_2 \cdots v_{m}.\] 
The   {\em alphabet} $A= \Alphabet(w)=\left\{w_i \mid 1\leq i \leq n\right\}$ is  equipped with a total order $\prec$ which extends lexicographically to the set $A^*$ of all words over $A$.

A word $u$ is called a {\em factor} of $w$ if $w = pus$ for some words $p,s$. The $i$-th prefix ending at position $i$ is denoted $w_p(i) =w_1w_2 \cdots w_{i}$ and the  $i$-th suffix starting at position $i$ is $w_{s}(i)=w_iw_{i+1} \cdots w_{n}$. Hence every word $w=w_1w_2 \cdots w_n$ factorizes as $w= w_p(i-1) w_{s}(i)$. Of course, $w_s(1)=w$ and $w_p(0)=\varepsilon$. The set of all length-$i$ factors of $w$ is denoted  $\fac_i(w)$ and the set of all factors of $w$ is $\fac(w)$. 

Two finite words $u$ and $v$ are {\em conjugates} when there exist words $x,y$ such that $u=xy$ and $v=yx$. The conjugacy class of a word $w$ is denoted by $[w]$. 
Thus, $[w]=\left\{v|v=w_s(i)w_p(i-1), i=1,2,\ldots, n\right\}$.

For any natural number $k$,  the {\em $k$-power} of a nonempty finite word $u$ is the concatenation of $k$ copies of $u$, and it is denoted by $u^k$. In particular, a \emph{square} is a word $w$ of the form $w=uu$. 
A word $w$ is said to be {\em primitive} if it is not a power greater than or equal to $2$ of a word distinct from $w$. The set of primitive factors is denoted by $\pri(w)$. For any word $u$ and any rational number $\alpha=\frac{m}{|u|}$, the $\alpha$-power of $u$ is defined to be $u^au_0$ where $u_0$ is a prefix of $u$, $a=\lfloor \alpha \rfloor$ is the integer part of $\alpha$, and $|u^au_0|=m$. The $\alpha$-power of $u$ is denoted by $u^{\alpha}$. A word $w$ is said to be {\em of period $n$} if there exists a word $u$ and a positive rational number $\alpha \geq 1$ such that $|u|=n$ and $w=u^{\alpha}$. The period $n$ is said to be {\em the smallest period} if for any prefix $u'$ of $w$ satisfying $|u'|<n$, $w$ is not a (rational) power of $u'$.

\begin{exa}{\rm 
Let $w=bbab$, then $[w]=[w]_4=\left\{bbab,babb,abbb,bbba\right\}$,\\ 
$w[2,4]=bab,\ w^{\frac94}=bbabbbabb$.}
\end{exa}

\begin{lem}[{\rm Fine and Wilf in~\cite{fiwi}}]
\label{Fine}
Let $w$ be a word having periods $k$ and $l$. If $|w|\ge k+l-\mathrm{gcd}(k,l)$ then $\mathrm{gcd}(k,l)$ is also a period of $w$.    
\end{lem}

\begin{cor}\label{lem:unique-primitive-root}
Let $x$ be a word. If $x=p^\alpha=q^\beta$, where \(p\) and \(q\) are primitive words and \(\alpha,\beta\ge 2\), then
\(p=q\).
\end{cor}

\begin{proof}
Since $|x|\ge \max\{2|p|,2|q|\}\ge |p|+|q|-\gcd(|p|,|q|)$,
by Lemma~\ref{Fine}, \(d=\gcd(|p|,|q|)\) is also a period of \(x\).
From the primitivity of $p$ and $q$, one has \(p=q\). \qed
\end{proof}

Next let us recall the formal definition of Rauzy graphs from~\cite{Rauzy}. Let $w$ be a word of length $n$. For any integer $i$ such that $1\leq i\leq n$, the Rauzy graph $\Gamma_i(w)$ is a directed graph whose set of vertices is $\fac_i(w)$ and the set of edges is $\fac_{i+1}(w)$.
An edge $e \in \fac_{i+1}(w)$ starts at the vertex $u$ and ends at the vertex $v$ if and only if $u$ is a prefix and $v$ is a suffix of $e$. 

The graphs $\Gamma_i(w)$ for different values of $i$ are pairwise disjoint; for convenience, define
\[
\Gamma(w)=\bigcup_{i=1}^{n}\Gamma_i(w).
\]
In general, a graph $M=(V', E')$ is called a \emph{subgraph} of $G$ if $V' \subset V$ and $E' \subset E$. In this case, let us denote $M \sqsubset G$.

For any primitive word $p \in \fac(w)$, if there exists an integer $i$ such that $i \geq |p|$ and that $[p]_i \subset \fac(w)$ and $[p]_{i+1} \subset \fac(w)$, then $C=([p]_{i},[p]_{i+1})$ is a subgraph of $\Gamma_i(w)$. From~\cite{sq}, $C=([p]_{i},[p]_{i+1})$ is called a {\em small circuit} of $\Gamma_i(w)$ and denoted by $C(p,i)$.

For $1\le i\le n$, let $sc_i(w)$ be the number of small circuits in $\Gamma_i(w)$, and set~\cite{sq}
\[
sc(w)=\sum_{i=1}^n sc_i(w),
\]
the total number of small circuits in $\Gamma(w)$.

\begin{lem}[{\rm Brlek and Li~\cite{sq}}]
\label{total-number}
Let $w$ be a finite word of length $n$. Then, 
\begin{itemize}
    \item[1.] for all integers $1\leq i\leq n$, one has 
$$sc_i\le |\fac_{i+1}(w)|-|\fac_i(w)|+1.$$
    \item[2.] the total number of small circuits in $\Gamma(w)$ is no more than $n$.
\end{itemize}
\end{lem}
\section{Some basic facts of rational power factors}\label{sec3}

Let $w$ be any word. For any $u\in\pri(w)$ such that $u^2\in\fac(w),|u|=\ell$, let 
$$\cl_w(u)=\{ p^m\in\fac(w)| p\in[u],m\in\ZZ,m\ge 2\};$$
$$\qcl_w(u)=\{ p^m\in\fac(w)| p\in[u],m\in\QQ,m\ge 2\};$$
$$D(w)=\{[u]|u^2\in\fac(w)\}.$$ 
Let $\rp(w)$ be the number of distinct rational power factors of $w$. By Corollary~\ref{lem:unique-primitive-root}, $\qcl_w(u)\cap\qcl_w(v)=\emptyset$ if and only if $u \not \in [v]$, hence
$$\rp(w)=\sum\limits_{[u]\in D(w)}|\qcl_w(u)|.$$
Let $$\rp(n)=\max_{|w|=n} \rp(w).$$
Let $M(u)$ be the number of small circuits in $\Gamma(w)$ of the form $C(u,m)$. In the following of this section, the factor $u$ is fixed, and we use $M$ instead of $M(u)$ for simplicity.\\ 
Let $S_{[u]}(w)=\{C(u,m)| C(u,m) \sqsubset \Gamma_m(w),\; 1\leq m\le|w|\}$ be the set of all small circuits in $\Gamma(w)$ of the form $C(u,m)$.
An easy observation is that if 
\(C(u,m)=([u]_m, [u]_{m+1}) \sqsubset \Gamma_m(w)\) for some $m > \ell$, then \(C(u,i)=([u]_{i}, [u]_{i+1}) \sqsubset \Gamma_{i}(w)\) for all $\ell \leq i\leq m$. Thus,
$$S_{[u]}(w)=\{C(u,\ell+i)|0\le i\le M(u)-1, i\in\ZZ\}.$$
Define a map $f$ from $S_{[u]}(w)$ to $2^{\fac(u^{\omega})}$ as follows:
$$f(C(u,\ell+i))=[u]_{2\ell+i},\quad 0\le i\le M-1, i\in\ZZ.$$
Hence for $i\ne j,\ f(C(u,\ell+i)) \cap f(C(u,\ell+j))= \emptyset$.
Let 
$$T=\bigcup_{i=0}^{M-1} f(C(u,\ell+i)).$$

\begin{lem}\label{no long factors}
For any integer $i\ge M$, $\qcl_w(u) \cap [u]_{2\ell+i}=\emptyset $.
\end{lem}

\begin{proof}
If there exists a word $v\in\qcl_w(u) \cap [u]_{2\ell+i_0}$ for some $i_0>M$, then its prefix $v_p(2\ell+M)\in\qcl_w(u) \cap [u]_{2\ell+M}$. Therefore, it is enough to prove in the case $i=M$. Given a word $p\in [u]_{2\ell+M}$, for any word $q\in [u]_{\ell+M+1}$, we have $q\in\fac(p)$. Hence if $p\in\qcl_w(u)$, then $C(u,\ell+M)$ is a small circuit in $\Gamma(w)$, which contradicts the definition of $S_{[u]}(w)$. \qed
\end{proof}

\begin{lem}\label{number of long rational factors}
For any integer $1\le i\le \ell$, $|\qcl_w(u) \cap [u]_{2\ell+M-i}|\le i$.
\end{lem}

\begin{proof}

If $i=\ell$, $|\qcl_w(u) \cap [u]_{2\ell+M-\ell}|\le |[u]_{\ell+M}| =\ell$. Now let $1\le i\le \ell-1$. Assume for contradiction that there exist $i+1$ distinct words
\[
u^{(1)},u^{(2)},\dots,u^{(i+1)} \in \qcl_w(u)\cap [u]_{2\ell+M-i}.
\]
Here we prove that for all $q\in [u]_{\ell+M+1}$, there exists $1 \leq t \leq i+1$ such that $q \in \fac(u^{(t)})$. Let $U=u^{\frac{3\ell+M}{\ell}}$, then all the words $u^{(1)},u^{(2)},\dots,u^{(i+1)}$ occur at least once in $U$ and $q$ occurs at least twice. For each $j\in\{1,\dots,i+1\}$, let $d_j\in\{1,2,\dots,\ell\}$ be the first occurrence of $u^{(j)}$ in $U$, and let $d_q\in\{1,2,\dots,\ell\}$ be the first occurrence of $q$ in $U$.
Define $$d_q^{(j)}=\begin{cases}
    d_q,&\; \text{if $d_j \leq d_q$};\\
    d_q+\ell,&\; \text{otherwise}.
\end{cases}$$
Then $0\le d_q^{(j)}-d_j <\ell$ are pairwise distinct. Thus, there exists an integer $1 \leq t\leq i+1$ such that $0\le d_q^{(t)}-d_t <\ell-i$. In this case, $q \in \fac(u^{(t)})$.\qed

\end{proof}

\begin{lem}\label{estimation of RP for one class}
If $M\le\ell$, then $|\qcl_w(u)|\le \frac12 M(M+1)$. If $M>\ell$, then $|\qcl_w(u)|\le \frac12 \ell(\ell+1)+\ell(M-\ell)$.
\end{lem}

\begin{proof}
By Lemma~\ref{no long factors}, one has $\qcl_w(u)\subset T$. If $M\le\ell$, then by Lemmas~\ref{no long factors} and \ref{number of long rational factors}, for $0\le i\le M$, one has $|\qcl_w(u) \cap [u]_{2\ell+M-i}|\le i$. Hence 
$$|\qcl_w(u)|\le \sum\limits_{i=0}^M i=\frac12 M(M+1).$$\\
If $M>\ell$, then 
$$|\qcl_w(u) \cap [u]_{2\ell+M-i}| \leq 
\begin{cases}
|[u]_{2\ell+M-i}|=\ell,& \; \text{if $\ell+1\le i\le M$};\\
i,& \; \text{if $1\le i\le \ell$}.\\
\end{cases}$$ 
The last inequality holds from Lemma~\ref{number of long rational factors}. Hence $$|\qcl_w(u)|\le \ell(M-\ell)+\sum\limits_{i=1}^{\ell} i=\frac12 \ell(\ell+1)+\ell(M-\ell).$$

\end{proof}

\section{Proof of the main theorem}\label{sec4}
In this section, we give an upper bound for the number of distinct rational power factors in a word. From now on, let $w$ be a word of length $n$ and let $N=n+1$;\\
let
\[
        D(w)=\{[u]|u\in\pri(w),u^2\in\fac(w)\}=\{[u_1],\ldots,[u_t]\};
\]
and let
\[
        \ell_i:=|u_i|,\qquad m_i:=M(u_i),\qquad A_i:=\ell_i+m_i.
\]
Without loss of generality, let us assume $A_1\ge A_2\ge\cdots\ge A_t$.\\
For each \(i\), define the {\em active interval} of \([u_i]\) by
\[
        I_i:=\{\ell_i,\ell_i+1,\ldots,A_i-1\}.
\]
Equivalently, \(q\in I_i\) if and only if the Rauzy graph \(\Gamma_q(w)\)
contains the small circuit \(C(u_i,q)\).\\
Let
\[
F(\ell,m):=
\begin{cases}
\left(m-\dfrac12\ell\right)\ell,& m\ge \ell,\\[1ex]
\dfrac12 m^2,& m<\ell.
\end{cases}
\]\\
Since we only concern the leading coefficient of the asymptotic upper bound of $\rp(n)$, we will only use the normalization of $F$, that is,  $F(\frac{\ell_i}{N},\frac{m_i}{N})$. In this sense, the domain of $F$ is $[0,1]^2$. 

\begin{lem}\label{estimation for RP in total}
One has 
$$\rp(w)\le N^2\sum_{i=1}^t F(\frac{\ell_i}{N},\frac{m_i}{N}) + O(N).$$ 
\end{lem}

\begin{proof}
For each $[u_i]\in D(w)$, its contribution
to $\rp(w)$ is bounded by $N^2F(\frac{\ell_i}{N},\frac{m_i}{N})+O(m_i)$. By Lemma~\ref{total-number}, one has
\begin{equation}\label{sum of M}
\sum_{i=1}^t m_i\le n<N.   
\end{equation}
Hence
$$\rp(w)\le N^2\sum_{i=1}^t F(\frac{\ell_i}{N},\frac{m_i}{N}) + O(N).$$ \qed
\end{proof}

\begin{lem}\label{lem:pivot-overlap}
For every \(1\le j\le t\), one has
\begin{equation}\label{lem9eq1}
\sum_{i=1}^t |I_i\cap I_j|\le N-2\ell_j .    
\end{equation}
In particular, taking \(j=1\), one has
\begin{equation}\label{lem9eq2}
\sum_{i=2}^t |I_i\cap I_1|\le N-2\ell_1-m_1 .    
\end{equation}
Moreover, for every \(1\le j\le t\), one has
\begin{equation}\label{lem9eq3}
2\ell_j+m_j\le N .    
\end{equation}
\end{lem}

\begin{proof}
Fix \(j\in\{1,\ldots,t\}\).  For each \(q\in I_j\), let
$k_q:=|\{i:q\in I_i\}|.$ Since \(q\in I_j\), one has
\(k_q\ge 1\).
For each $i$ such that $q\in I_i$, $C(u_i,q)$ is a small circuit in $\Gamma_q(w)$ and they are distinct. Hence
\[
        k_q\le sc_q(w).
\]
By Lemma~\ref{total-number},
\[
        sc_q(w)\le |\fac_{q+1}(w)|-|\fac_q(w)|+1.
\]
Thus,
\[
        k_q-1
        \le
        |\fac_{q+1}(w)|-|\fac_q(w)|.
\]
Summing over
\(
        q=\ell_j,\ell_j+1,\ldots,A_j-1,
\) we get
\[
\begin{aligned}
        \sum_{q\in I_j}(k_q-1)
        &\le
        \sum_{q=\ell_j}^{A_j-1}
        \bigl(|\fac_{q+1}(w)|-|\fac_q(w)|\bigr)\\
        &=
        |\fac_{A_j}(w)|-|\fac_{\ell_j}(w)|.
\end{aligned}
\]
The left-hand side equals
\[
        \sum_{i\ne j}|I_i\cap I_j|.
\]
For the right-hand side, we use
\[
        |\fac_{A_j}(w)|\le n-A_j+1=N-A_j
\]
and
\[
        |\fac_{\ell_j}(w)|\ge |[u_j]_{\ell_j}|=\ell_j.
\]
Hence
\[
        \sum_{i\ne j}|I_i\cap I_j|
        \le
        N-A_j-\ell_j.
\]
Since \(A_j=\ell_j+m_j\), one has
\[
        \sum_{i\ne j}|I_i\cap I_j|
        \le
        N-2\ell_j-m_j.
\]
Adding the term \(|I_j\cap I_j|=|I_j|=m_j\) to both sides yields
\[
        \sum_{i=1}^t |I_i\cap I_j|
        \le
        N-2\ell_j.
\]
To prove Equation~\ref{lem9eq2}, taking \(j=1\), so that
\[
        \sum_{i=2}^t |I_i\cap I_1|
        \le
        N-2\ell_1-m_1.
\]
To prove Equation~\ref{lem9eq3}, taking \(i=j\), so that
\[
        m_j\le N-2\ell_j,
\]
\qed
\end{proof}

We now compute the upper bound of the quadratic coefficient of $\rp(w)$. It is useful to introduce the following function: for
\(0<s\le 1/2\) and \(s\le B\le 2s\), define 
$$\mathcal P(s,B):=\max \sum_{\substack{x_j+y_j\le s\\\sum_j y_j\le B\\x_j,y_j\ge 0}} F(x_j,y_j) $$ 

To give a concrete illustration of the parameters $x_i, y_i$ in the previous equation, for $[u_j] \in D(w)$, \(x_j=\frac{\ell_j}{N}\) and  \(y_j=\frac{m_j}{N}\).

\paragraph{Sketch of proof.}
We start from
\[
\rp(w)\le \sum_i N^2F(\frac{\ell_i}{N},\frac{m_i}{N})+O(n),
\qquad
\sum_i m_i\le n.
\]
There are two cases.

If \(A_1\le N/2\), then every active interval ends before \(N/2\). We will prove the leading coefficient of $\rp(w)$ is no more than $\mathcal P\left(\frac12,1\right)$.
Lemma~\ref{lem:packing-function} gives
\(
\mathcal P\left(\frac12,1\right)\le \frac18.
\)

If \(A_1>N/2\), we compute $\rp(w)$ in three parts. Let $D_1=\{[u_1]\}$, $D_2=\{[u_i]||u_i|<|u_1|\}$ and $D_3=\{[u_i]||u_i|\ge |u_1|, i \ne 1\}$ and let $\rp_i(w)=|\cup \{\qcl_w(u)| [u] \in D_i \}|$ for $i=1,2,3$, we will upper bound $\rp_i(w)$ by $F(\frac{\ell_1}{N},\frac{m_1}{N})$, $\mathcal P(\frac{N-\ell_1-m_1-K}{N},\frac{N-m_1-K}{N})$ and $\frac12(\frac{K}{N})^2$ for $i=1,2,3$ respectively with some suitable number $K$. Then we will conclude by using Lemma~\ref{lem:one-eighth-reduction}.

In summary, in each case, we can transform the problem into a constrained multi-variable optimization problem. To do so, we have to check if each case satisfies the constraints of the corresponding lemma.

\begin{thm}\label{thm:rational-power-eighth}
\[
\rp(w)\le \frac18 n^2+O(n).
\]
\end{thm}

\begin{proof}
As mentioned in the {\em sketch of proof}, we optimize $\rp(w)$ in two cases.

\noindent
\textbf{Case 1: \(A_1\le N/2\).}\\
Then one has
\(
\frac{\ell_i}{N}+\frac{m_i}{N}=\frac{A_i}{N}\le \frac{A_1}{N}\le\frac12.
\)
Moreover, from Equation~\ref{sum of M}, $\sum_{i=1}^t \frac{m_i}{N}\le 1$. Therefore, by the definition of \(\mathcal P\) and Lemma~\ref{lem:packing-function},
\[
\sum_{i=1}^t F(\frac{\ell_i}{N},\frac{m_i}{N})
\le
\mathcal P\left(\frac12,1\right)\le \frac18.
\]

\medskip

\noindent
\textbf{Case 2: \(A_1>N/2\).}
The contribution of $D_1(w)$ is
\[
\rp_1(w)=F(\ell_1,m_1)+O(n)=N^2F(\frac{\ell_1}{N},\frac{m_1}{N})+O(n).
\]\\
Now we compute $\rp_3(w)$. In this case, for all $[u_i] \in D_3(w)$, $$\ell_1 \le \ell_i \le A_i \le A_1.$$ Thus, $$m_i=|I_i \cap I_1|.$$ Moreover, from the definition of the function $F$, one has $$F(\frac{\ell}{N},\frac{m}{N})\le \frac12 \left(\frac{m}{N}\right)^2.$$
Define
\[
K:=\sum_{[u_i]\in D_3}|I_i\cap I_1|.
\]
One has
\[
\sum_{[u_i]\in D_3}F(\frac{\ell_i}{N},\frac{m_i}{N})
\le
\frac12\sum_{[u_i]\in D_3}\left(\frac{m_i}{N}\right)^2
=
\frac12\sum_{[u_i]\in D_3}\left(\frac{|I_i\cap I_1|}{N}\right)^2
\le
\frac12\left(\frac{K}{N}\right)^2.
\]
Next we compute $\rp_2(w)$. We claim
\[
\sum_{[u_i]\in D_2}F(\frac{\ell_i}{N},\frac{m_i}{N})
\le
\mathcal P(\frac{N-\ell_1-m_1-K}{N},\frac{N-m_1-K}{N}).
\]
To verify this, it is enough to check
\begin{equation}\label{eq:cond-li-mi}
    \frac{\ell_i+m_i}{N}
    \le
    \frac{N-\ell_1-m_1-K}{N},\ \forall [u_i]\in D_2,
\end{equation}

\begin{equation}\label{eq:cond-sum-mi}
    \frac1N\sum_{[u_i]\in D_2} m_i
    \le
    \frac{N-m_1-K}{N},
\end{equation}

\begin{equation}\label{eq:cond-half}
    0
    \le
    \frac{N-\ell_1-m_1-K}{N}
    \le
    \frac{1}{2},
\end{equation}

\begin{equation}\label{eq:cond-comparison}
    \frac{N-\ell_1-m_1-K}{N}
    \le
    \frac{N-m_1-K}{N}
    \le
    2\,\frac{N-\ell_1-m_1-K}{N}.
\end{equation}
We first verify Inequality~\ref{eq:cond-li-mi}. By Equation~\ref{lem9eq2}, we have $K\leq N-2\ell_1-m_1$.
For any $[u_i]\in D_2$, one has $A_i\leq A_1$ and $\ell_i<\ell_1$. If $A_i\leq \ell_1$, then
$A_i\leq \ell_1\leq N-\ell_1-m_1-K$, where the last Inequality follows from Equation~\ref{lem9eq2} also. If $A_i>\ell_1$, then
$I_i\cap I_1=\{\ell_1,\ell_1+1,\ldots,A_i-1\}$. Thus
$A_i-\ell_1=|I_i\cap I_1|\leq N-2\ell_1-m_1-K$, and again $A_i\leq N-\ell_1-m_1-K$. Therefore, in all cases, Inequality~\ref{eq:cond-li-mi} holds.\\
Next, if $[u_i]\in D_3$, then $\ell_i\geq \ell_1$ and $A_i\leq A_1$, hence $I_i\subseteq I_1$ and $|I_i\cap I_1|=m_i$. Therefore $K=\sum_{[u_i]\in D_3}m_i$. From Equation~\ref{sum of M}, we get
\[
\frac1N\sum_{[u_i]\in D_2}m_i
\leq \frac{N-m_1-\sum_{[u_i]\in D_3}m_i}{N}
= \frac{N-m_1-K}{N},
\]
which proves Inequality~\ref{eq:cond-sum-mi}.\\
Third, from Equation~\ref{lem9eq2}, we have $N-\ell_1-m_1-K\geq \ell_1>0$. Since we are in Case 2, $A_1=\ell_1+m_1>N/2$, and hence $N-\ell_1-m_1-K<N/2$. This proves inequality~\ref{eq:cond-half}.\\
Lastly, the right inequality of \ref{eq:cond-comparison} is equivalent to $2\ell_1+m_1+K\leq N$, which follows again from Equation~\ref{lem9eq2}.\\
Combining the three estimates gives
\[
\sum_{i=1}^t F(\frac{\ell_i}{N},\frac{m_i}{N})
\le
F(\frac{\ell_1}{N},\frac{m_1}{N})+\mathcal P(\frac{N-\ell_1-m_1-K}{N},\frac{N-m_1-K}{N})+\frac12(\frac{K}{N})^2
.
\]
By Lemma~\ref{lem:one-eighth-reduction},
\[
N^2\sum_{i=1}^tF(\frac{\ell_i}{N},\frac{m_i}{N})
\le
\frac18N^2.
\]
Finally it is enough to check the two conditions of Lemma~\ref{lem:one-eighth-reduction}, that is,

\begin{equation}\label{cond1}
0< \frac{\ell_1}{N}\le \frac{N-\ell_1-m_1-K}{N}\le \frac12,   
\end{equation}

\begin{equation}\label{cond2}
0\le \frac{K}{N}\le\frac12-\frac{N-\ell_1-m_1-K}{N}.
\end{equation}
The inequality $0\le\frac{N-\ell_1-m_1-K}{N}\le\frac{1}{2}$ is from Equation~\ref{eq:cond-half}; $\frac{\ell_1}{N}\le \frac{N-\ell_1-m_1-K}{N}$ is obtained by Equation~\ref{lem9eq2} in Lemma~\ref{lem:pivot-overlap} and Inequality~\ref{cond2} by $A_1>\frac N2$. \qed
\end{proof}

We now prove the two elementary optimization lemmas used in the proof of
Theorem~\ref{thm:rational-power-eighth}.

\begin{lem}\label{lem:packing-function}
Let \(0<s\le 1/2\) and \(s\le B\le 2s\). Then
\[
\mathcal P(s,B)\le s^2V(B/s),
\]
where, for \(1\le T\le 2\),
\[
V(T):=
\begin{cases}
\dfrac{(2-T)(3T-2)}{4},
& 1\le T\le \dfrac43,\\[8pt]
\dfrac32T^2-4T+3,
& \dfrac43\le T\le \dfrac32,\\[8pt]
\dfrac12(3-T)(T-1),
& \dfrac32\le T\le 2.
\end{cases}
\]
In particular,
\[
\mathcal P\left(\frac12,1\right)\le \frac18.
\]
\end{lem}

\begin{proof}
Let $T:=\frac{B}{s}$, then \(1\le T\le 2\).

Fix one pair \((x,y)\) with $x+y\le s$. Write
\[
y=rs,
\qquad 0\le r\le 1.
\]
For a fixed \(y\), the largest possible value of \(F(x,y)\), under the condition
\(0\le x\le s-y\), is $s^2\varphi(r)$, where
\[
\varphi(r):=
\begin{cases}
\dfrac12r^2, & 0\le r\le \dfrac12,\\[6pt]
\dfrac12(1-r)(3r-1), & \dfrac12\le r\le 1.
\end{cases}
\]
Indeed, if \(r\le 1/2\), then one can choose \(x\ge y\), and the maximum is
\[
\frac12y^2=s^2\frac12r^2.
\]
If \(r\ge 1/2\), then \(x\le s-y\le y\), and
\[
F(x,y)=xy-\frac12x^2
\]
is increasing in \(x\) on \(0\le x\le s-y\).  Hence the maximum is attained at
\(x=s-y\), and equals
\[
(s-y)y-\frac12(s-y)^2
=
s^2\frac12(1-r)(3r-1).
\]
Therefore, for every admissible family \((x_j,y_j)\), if $r_j:=\frac{y_j}{s}$, then
\[
F(x_j,y_j)\le s^2\varphi(r_j),\quad\sum_j r_j\le T.
\]
It remains to maximize $\sum_j\varphi(r_j)$ under the constraints
\[
0\le r_j\le 1,
\qquad
\sum_j r_j\le T.
\]
Let $r=T-\sum_jr_j$, then $\max \sum_j\varphi(r_j) \leq \max (\sum_j\varphi(r_j)+\varphi(r))$. We can now compute $ \max \sum_j\varphi(r_j)$ under the restrictions:
\[
0\le r_j\le 1,
\qquad
\sum_j r_j= T.
\]
We remark that for any pair $0 \le r_i, r_j<\frac{1}{2}$ with $r_i \ne r_j$, if $r_i+r_j \leq \frac{1}{2}$, then $$\varphi(r_i)+\varphi(r_j) \leq \varphi(r_i+r_j);$$ if $r_i+r_j > \frac{1}{2}$, then $$\varphi(r_i)+\varphi(r_j) \leq \varphi(\frac{1}{2})+ \varphi(r_i+r_j-\frac{1}{2}).$$ Thus, one can reduce the number of variables in objective function $\sum_j\varphi(r_j)$ into at most $4$. That is, $\max \sum_j\varphi(r_j) \le \max \sum_{j=1}^k\varphi(r'_j)$ under the restrictions: 
\[
0\le r'_1 \le r'_2 \le \ldots \le r'_k\le 1
,\qquad
\sum_{j=1}^k r'_j= T
,\qquad r'_2 \ge \frac{1}{2}.
\]
Moreover, since $\varphi$ is concave in $[\frac{1}{2},1]$, 
$$\sum_{j=2}^k\varphi(r'_j) \leq (k-1)\varphi(\frac{1}{k-1}\sum_{j=2}^kr'_j).$$
Hence, $$\sum_{j=1}^k\varphi(r'_j) \leq \varphi(r'_1)+(k-1)\varphi(\frac{T-r'_1}{k-1}).$$
For $k=2,3,4$, a direct one-variable calculation gives the following three possible maxima:
\[
2\varphi(T/2)
\quad
\text{for }
1\le T\le \frac43,
\]
\[
2\varphi(2-T)+\frac12(3T-4)^2
\quad
\text{for }
\frac43\le T\le \frac32,
\]
and
\[
3\varphi(T/3)
\quad
\text{for }
\frac32\le T\le 2.
\]
These are exactly
\[
V(T)=
\begin{cases}
\dfrac{(2-T)(3T-2)}{4},
& 1\le T\le \dfrac43,\\[8pt]
\dfrac32T^2-4T+3,
& \dfrac43\le T\le \dfrac32,\\[8pt]
\dfrac12(3-T)(T-1),
& \dfrac32\le T\le 2.
\end{cases}
\]\qed
\end{proof}

\begin{lem}\label{lem:one-eighth-reduction}
Let \(0<a\le s\le 1/2\), \(0\le y\le 1/2-s\).
Then
\[
F(a,1-a-s-y)+\mathcal P(s,s+a)+\frac12y^2\le \frac18.
\]
\end{lem}

\begin{proof}
By Lemma~\ref{lem:packing-function},
\[
\mathcal P(s,s+a)\le s^2V\left(1+\frac{a}{s}\right).
\]
For brevity, let
\[
W(a,s):=s^2V\left(1+\frac{a}{s}\right).
\]
From the explicit formula for \(V\), we get
\[
W(a,s)=
\begin{cases}
\dfrac14s^2+\dfrac12as-\dfrac34a^2,
& 0\le \frac{a}{s}\le \frac{1}{3},\\[8pt]
\dfrac12s^2-as+\dfrac32a^2,
&\frac{1}{3}\le \frac{a}{s}\le \frac{1}{2},\\[8pt]
as-\dfrac12a^2,
& \frac{1}{2}\le \frac{a}{s}\le 1.
\end{cases}
\]
It is enough to prove
\[
F(a,1-a-s-y)+W(a,s)+\frac12y^2\le \frac18.
\]
Let
\(S=s+y.\)
Then
\[
0<a\le s\le S\le \frac12.
\]
For fixed \(a\) and \(S\), we maximize $W(a,s)+\frac12(S-s)^2$ over $a\le s\le S$. Then 
one has
\begin{align*}
W(a,s)+\frac12(S-s)^2\le\max
\begin{cases}
U_1(a,S)=\frac12a^2+\frac12(S-a)^2,\quad &0\le a\le S\le \dfrac12,\\ 
U_2(a,S)=\frac32a^2+\frac12(S-2a)^2,\quad &0\le 2a\le S\le \dfrac12,\\
U_3(a,S)=3a^2+\frac12(S-3a)^2,\quad &0\le 3a\le S\le \dfrac12,\\
U_4(a,S)=W(a,S), \quad &0\le a\le S\le \dfrac12.
\end{cases}
\end{align*}
Thus it remains to prove
\[
F(a,1-a-S)+U_j(a,S)\le \frac18,\quad j=1,2,3,4.
\]
Recall that $$F(a,1-a-S)=\begin{cases}
    \frac{(1-a-S)^2}{2},& \; \text{if $1-a-S\le a$},\\
    a(1-a-S- \frac{a}{2}),& \; \text{if $1-a-S> a$}.
\end{cases}$$
One has
$$\max_{0<a\le S \le \frac{1}{2} } F(a,1-a-S)+U_j(a,S)=\begin{cases}
    \max_{i=1,2,3,4} \{U_i+F(a,1-a-S)\},& \text{if $3a \le S<\frac{1}{2}$},\\
    \max_{i=1,2,4} \{U_i+F(a,1-a-S)\},& \text{if $2a \le S<3a$},\\
    \max_{i=1,2} \{U_i+F(a,1-a-S)\},& \text{if $a \le S<2a$}.\\
\end{cases}$$
One can check in all cases$$\max_{0<a\le S \le \frac{1}{2} } F(a,1-a-S)+U_j(a,S) \le \frac{1}{8}.$$ 
Therefore
\[
F(a,1-a-s-y)+W(a,s)+\frac12y^2\le \frac18.
\] \qed
\end{proof}

\paragraph{Extremal profiles allowed by the relaxation.}
The preceding proof also indicates where the constant \(1/8\) comes from.

In Case 1, we achieve the maximum value if the two following conditions hold: 
\[
A_i=\ell_i+m_i\approx \frac N2 \; \text{for all $[u_i] \in D(w)$}
\]
and
\[
\sum_i m_i\approx N.
\]
Two typical equality profiles are:
\[
(\ell_i,m_i)\approx \left(\frac N6,\frac N3\right)
\quad\text{for three classes,}
\]
or
\[
(\ell_i,m_i)\approx \left(\frac N4,\frac N4\right)
\quad\text{for four classes.}
\]

In Case 2,  we achieve the maximum value when: 
\[
(\ell_1,m_1)\approx \left(\frac N4,\frac N2\right)
\]
and 
\[
D_3=\emptyset, \quad\text{and for any }[u_i]\in D_2,\  A_i+\ell_i \approx \ell_1.
\]
A typical relaxed profile is
\[
(\ell_i,m_i)\approx \left(\frac N{12},\frac N6\right)
\quad\text{for three classes,}
\]
or
\[
(\ell_i,m_i)\approx \left(\frac N8,\frac N8\right)
\quad\text{for four classes.}
\]

\section{A lower bound of $\rp(n)$}

In this section, we give the first estimation of lower bound of $\rp(n)$ by proving that there exist infinitely many words $w$ such that $\rp(w)\ge \frac{|w|^2}{9}-O(|w|)$. For $n \in \N$, define 
$$w_n=(a^nba^{n-1}b)^4a^{n-1}.$$
One can check $|w_n|=9n+3.$ Let us compute $\rp(w)$.\\
Since $a^n \in \fac(w_n)$, 
$$\qcl_{w_n}(a)=\{a^k|2 \le k \le n \}.$$
Since $v_1=(a^{n-1}ba)^2a^{n-2} \in \fac(w_n)$, 
$$\qcl_{w_n}(a^{n-1}ba)=\{v_1[i,j] |1 \le i <j \le 3n, j-i \geq 2n+1 \}.$$
Since $v_2=(a^{n-1}b)^2a^{n-1} \in \fac(w_n)$, 
$$\qcl_{w_n}(a^{n-1}b)=\{v_2[i,j] |1 \le i <j \le 3n-1, j-i \geq 2n-1 \}.$$
Moreover,
\begin{align*}
 \qcl_{w_n}(a^{n-1}ba^nb)
 &=
 \{w_n[i,j]|1\le i\le 2n+1,4n+2\le j\le 6n+2,j-i\ge 4n+1\} \\
 &\cup
 \{w_n[i,j]|1\le i\le 2n+1,6n+3\le j\le 9n+3\}.
\end{align*}
Thus, \begin{align*}
    \rp(w_n)&=|\qcl_{w_n}(a)|+|\qcl_{w_n}(a^{n-1}ba)|+|\qcl_{w_n}(a^{n-1}b)|\\
    &+|\qcl_{w_n}(a^{n-1}ba^nb)|\\
    &=(n-1)+\frac{n(n-1)}{2}+\frac{n(n+1)}{2}\\
    &+\frac{(2n+1)(2n+2)}{2}+(2n+1)(3n+1)\\
    &=9n^2+9n+1=\frac{|w|^2}{9}-O(|w|).
    \end{align*}

\section{Conclusion}
In this work, we prove that $\rp(n)$ the maximum number of distinct rational powers in a word of length $n$ satisfies $\frac{n^2}{9}-O(n) \le \rp(n) \le \frac{n^2}{8}$. However, a computer simulation suggests that $\rp(n) \sim 0.115n^2$. Thus, both upper bound and lower bound of $\rp(n)$ may be improved.

\bibliographystyle{splncs03}
\bibliography{biblio}

\begin{thebibliography}{1}
\providecommand{\url}[1]{\texttt{#1}}
\providecommand{\urlprefix}{URL }

\bibitem{sq}
Brlek, S., Li, S.: On the number of squares in a finite word. Combinatorial Theory  5(1), ~3 (2025)

\bibitem{CMFS10}
Crochemore, M., Fazekas, S.Z., Iliopoulos, C.S., Jayasekera, I.: Number of occurrences of powers in strings. International Journal of Foundations of Computer Science  21(4),  535--547 (2010)

\bibitem{fiwi}
Fine, N.J., Wilf, H.S.: Uniqueness theorems for periodic functions. Proceedings of the American Mathematical Society  16(1),  109--114 (1965)

\bibitem{FraenkelS98}
Fraenkel, A.S., Simpson, J.: How many squares can a string contain? J. Comb. Theory, Ser. {A}  82(1),  112--120 (1998)

\bibitem{kpower}
Li, S., Pachocki, J., Radoszewski, J.: A note on the maximum number of {$k$}-powers in a finite word. The Electronic Journal of Combinatorics  31(3),  P3.14 (2024), \url{https://doi.org/10.37236/11270}

\bibitem{lothaire3}
Lothaire, M.: Applied Combinatorics on Words. Cambridge University Press, Cambridge (2005)

\bibitem{OS24}
Parshina, O., Puzynina, S.: Finite and infinite closed-rich words. Theoretical Computer Science  984,  114315 (2024)

\bibitem{Rauzy}
Rauzy, G.: Suites à termes dans un alphabet fini. Seminar on Number Theory (Univervité de Bordeaux I, Talence)  25,  1--16 (1983)

\bibitem{Thue1906}
Thue, A.: {\"U}ber unendliche zeichenreihen. Norske Videnskabers Selskabs Skrifter Mat.-Nat. Kl.  7,  1--22 (1906)

\end{thebibliography}

\end{document}